\documentclass[letterpaper, 10 pt, journal]{ieeeconf}  % Comment this line out if you need a4paper

\IEEEoverridecommandlockouts                              % This command is only needed if
\usepackage{amsmath} % assumes amsmath package installed
\usepackage{amssymb}  % assumes amsmath package installed
\usepackage{mathrsfs}
\usepackage{graphics} % for pdf, bitmapped graphics files
\usepackage{subfig}
\usepackage{epsfig} % for postscript graphics files

\newcommand{\Real}{\mathbb R}

\newtheorem{definition}{\textbf{Definition}}%[chapter]
\newtheorem{assumption}{\textbf{Assumption}}
\newtheorem{remark}{\textbf{Remark}}%[chapter]
\newtheorem{theorem}{\textbf{Theorem}}%[chapter]

\title{\LARGE \bf
Hierarchical and Safe Motion Control for Cooperative Locomotion of Robotic Guide Dogs and Humans: A Hybrid Systems Approach}

\author{Kaveh Akbari Hamed$^{1}$, Vinay R. Kamidi$^{1}$, Wen-Loong Ma$^{2}$, Alexander Leonessa$^{1}$, and Aaron D. Ames$^{2}$% <-this % stops a space
\thanks{*The work of K. Akbari Hamed amd V. R. Kamidi is supported by the National Science Foundation (NSF) under Grant Number 1637704/1854898. The work of A. D. Ames is supported by the NSF under Grant Numbers 1544332, 1724457, and 1724464 as well as Disney Research LA. The content is solely the responsibility of the authors and does not necessarily represent the official views of the NSF.}% <-this % stops a space
\thanks{$^{1}$K. Akbari Hamed, V. R. Kamidi, and A. Leonessa are with the Department of Mechanical Engineering, Virginia Tech, Blacksburg, VA 24061 USA {\tt\small kavehakbarihamed@vt.edu}, {\tt\small vinay28@vt.edu} and {\tt\small aleoness@vt.edu}}%
\thanks{$^{2}$W. Ma and A. D. Ames are with the Department of Mechanical and Civil Engineering, California Institute of Technology, Pasadena, CA 91125 USA {\tt\small wma@caltech.edu} and {\tt\small ames@cds.caltech.edu}}%
}

\begin{document}

\maketitle
\thispagestyle{empty}
\pagestyle{empty}

%%%%%%%%%%%%%%%%%%%%%%%%%%%%%%%%%%%%%%%%%%%%%%%%%%%%%%%%%%%%%%%%%%%%%%%%%%%%%%%%

\begin{abstract}
This paper presents a hierarchical control strategy based on hybrid systems theory, nonlinear control, and safety-critical systems to enable cooperative locomotion of robotic guide dogs and visually impaired people. We address high-dimensional and complex hybrid dynamical models that represent collaborative locomotion. At the high level of the control scheme, local and nonlinear baseline controllers, based on the virtual constraints approach, are designed to induce exponentially stable dynamic gaits. The baseline controller for the leash is assumed to be a nonlinear controller that keeps the human in a safe distance from the dog while following it. At the lower level, a real-time quadratic programming (QP) is solved for modifying the baseline controllers of the robot as well as the leash to avoid obstacles. In particular, the QP framework is set up based on control barrier functions (CBFs) to compute optimal control inputs that guarantee safety while being close to the baseline controllers. The stability of the complex periodic gaits is investigated through the Poincar\'e return map. To demonstrate the power of the analytical foundation, the control algorithms are transferred into an extensive numerical simulation of a complex model that represents cooperative locomotion of a quadrupedal robot, referred to as Vision 60, and a human model. The complex model has $16$ continuous-time domains with $60$ state variables and $20$ control inputs.
\end{abstract}

%%%%%%%%%%%%%%%%%%%%%%%%%%%%%%%%%%%%%%%%%%%%%%%%%%%%%%%%%%%%%%%%%%%%%%%%%%%%%%%%

\vspace{-1.2em}
\section{INTRODUCTION}
\label{INTRODUCTION}
\vspace{-0.5em}

This paper aims to develop an analytical foundation, based on hybrid systems theory, nonlinear control, quadratic programming, and safety-critical systems, to develop a hierarchical control algorithm that enables safe and stable cooperative locomotion of robotic guide dogs and visually impaired people (see Fig. \ref{v60_human}). One of the most challenging problems in deploying autonomous guide robots is to enable  \textit{ubiquitous mobility}. More than half the Earth's landmass is inaccessible to wheeled vehicles which motivates the deployment of intelligent and highly agile \textit{legged guide robots} to access these environments. In particular, infrastructures for human-centered communities, including factories, offices, and homes, are developed for humans which are bipedal walkers capable of stepping over gaps and walking up/down stairs.

%%%%%%%%%%%%%%%%%%%%%%%%%%%%%%%%%%%%%%%%%%%%%%%%%%%%%%%%%%%%%%%%%%%%%%%%%%%%%%%%

\begin{figure}[t!]
\centering
\includegraphics[width=\linewidth]{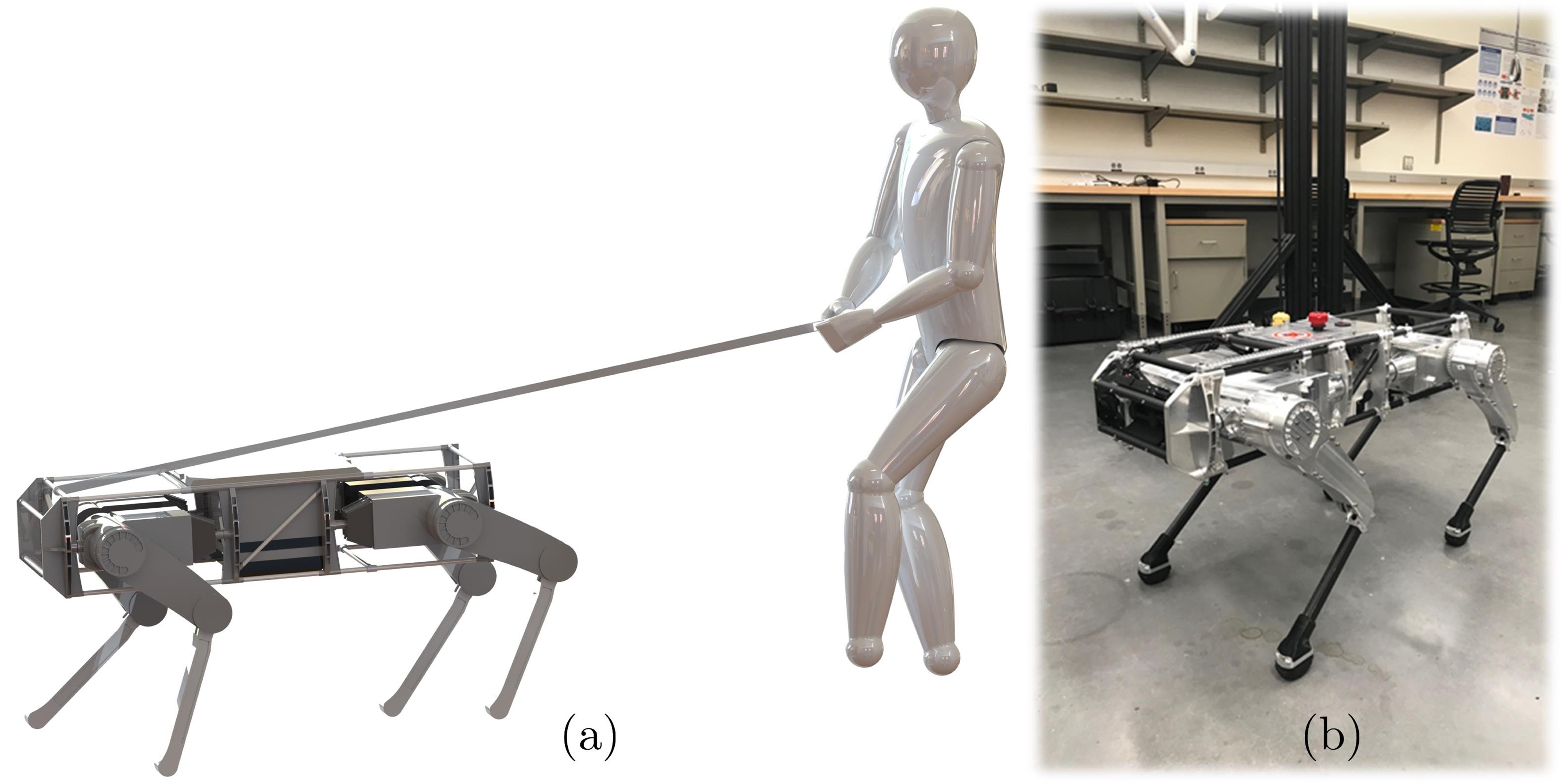}
\vspace{-1.8em}
\caption{(a) Illustration of a visually impaired human being guided by a quadrupedal assistance robot. (b) Vision 60 robot manufactured by Ghost Robotics \cite{Ghost_Robotics} whose full-order hybrid model will be used for the numerical simulations.}
\label{v60_human}
\vspace{-1.8em}
\end{figure}

%%%%%%%%%%%%%%%%%%%%%%%%%%%%%%%%%%%%%%%%%%%%%%%%%%%%%%%%%%%%%%%%%%%%%%%%%%%%%%%%

\vspace{-1em}
\subsection{Related Work}
\vspace{-0.3em}

Although important theoretical and technological advances have occurred for the construction and control of guide robots, state-of-the-art approaches are mainly tailored to the deployment of wheeled vehicles and \textit{not} legged guide robots (e.g., \cite{blind_01,blind_02,blind_03}). Unlike wheeled guide robots, legged robots are \textit{inherently unstable} complex dynamical systems with hybrid nature and high degrees of freedom (DOF). This complicates the design of feedback control algorithms that ensure stable and safe cooperative locomotion of guide dogs and human. Hybrid systems theory has become a powerful approach for modeling and control of legged robots both in theory and practice \cite{Grizzle_Asymptotically_Stable_Walking_IEEE_TAC,Westervelt_Grizzle_Koditschek_HZD_IEEE_TRO,Chevallereau_Grizzle_3D_Biped_IEEE_TRO,Ames_RES_CLF_IEEE_TAC,Ames_DURUS_TRO,Sreenath_Grizzle_HZD_Walking_IJRR,Park_Grizzle_Finite_State_Machine_IEEE_TRO,Poulakakis_Grizzle_SLIP_IEEE_TAC,Tedrake_Robus_Limit_Cycles_CDC,Byl_HZD,Johnson_Burden_Koditschek,Spong_Controlled_Symmetries_IEEE_TAC,Manchester_Tedrake_LQR_IJRR,Vasudevan2017}. Existing nonlinear control approaches that address the hybrid nature of legged locomotion models are developed based on hybrid reduction \cite{Ames_HybridReduction_Original_Paper}, controlled symmetries \cite{Spong_Controlled_Symmetries_IEEE_TAC}, transverse linearization \cite{Manchester_Tedrake_LQR_IJRR}, and hybrid zero dynamics (HZD) \cite{Westervelt_Grizzle_Koditschek_HZD_IEEE_TRO,Ames_RES_CLF_IEEE_TAC}. State-of-the art nonlinear control approaches for dynamic legged locomotion have been tailored to stable locomotion of legged robots, but \textit{not} stable and safe cooperative locomotion of legged guide robots and visually impaired people.

%%%%%%%%%%%%%%%%%%%%%%%%%%%%%%%%%%%%%%%%%%%%%%%%%%%%%%%%%%%%%%%%%%%%%%%%%%%%%%%%

\begin{figure*}[t!]
\centering
\includegraphics[width=\linewidth]{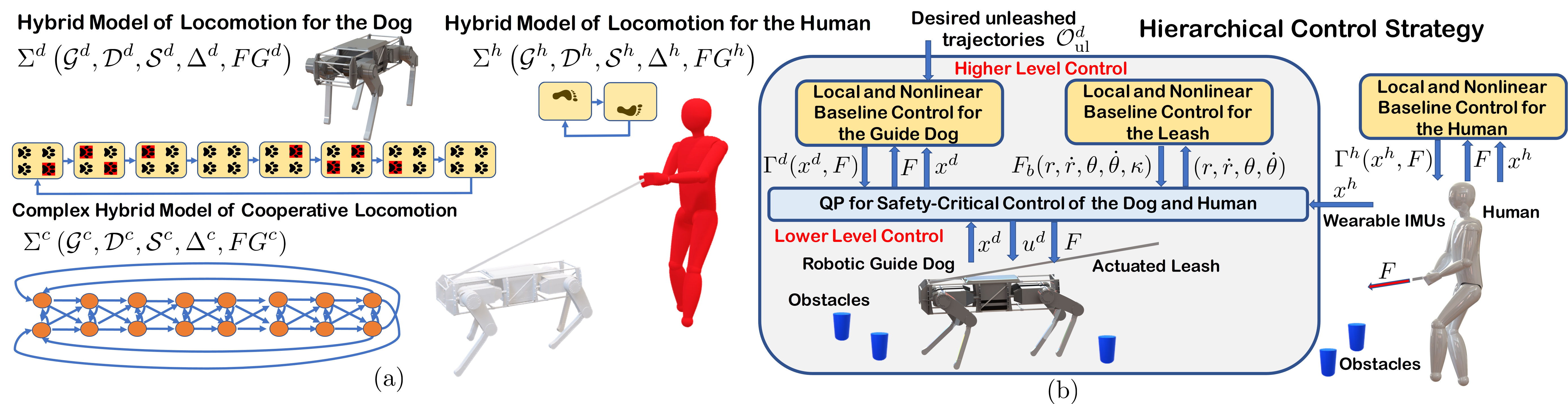}
\vspace{-1.8em}
\caption{(a) Illustration of the hybrid models for the unleashed and leashed locomotion of the guide robot and human. (b) Illustration of the proposed hierarchical control strategy for the safe and stable cooperative locomotion.}
\label{Illustration_control_scheme}
\vspace{-1.7em}
\end{figure*}

%%%%%%%%%%%%%%%%%%%%%%%%%%%%%%%%%%%%%%%%%%%%%%%%%%%%%%%%%%%%%%%%%%%%%%%%%%%%%%%%

\vspace{-1em}
\subsection{Objectives and Contributions}
\vspace{-0.2em}

The \textit{objectives} and \textit{contributions} of this paper are to present a formal foundation towards 1) developing complex hybrid models of cooperative locomotion of legged guide dogs and human, and 2) creating a hierarchical control algorithm, based on nonlinear control, quadratic programming, and control barrier functions (CBFs) \cite{ames2017cbf,gurriet2018online,nguyen20163d}, to ensure stability, safety, and obstacle avoidance. We address complex and high-dimensional models of cooperative legged locomotion via hybrid systems approach. An actuated leash structure is considered for the coordination of the dog and human locomotion while steering the human for safety and obstacle avoidance. At the higher level, the proposed hierarchical control algorithm employs local and nonlinear controllers, referred to as baseline controllers, that induce asymptotically stable unleashed locomotion patterns for the robotic dog and human. The baseline controllers are synthesized via the HZD approach and assumed to have access to the local state measurements as well as the force measurement applied by the leash structure. The leash baseline controller is then designed to keep the human in a safe distance from the robot while following it. The existence and stability of complex and leashed locomotion patterns for the coupled dynamics are addressed through the Poincar\'e return map. At the lower level of the control strategy, the baseline controllers for the dog and leash are modified by a real-time quadratic programming (QP) that includes CBF constraints to ensure safety and obstacle avoidance. The power of the anlytical results are demonstrated  on an extensive numerical simulation of a complex hybrid model that represents cooperative locomotion of a quadrupedal robot, referred to as Vision 60 \cite{Ghost_Robotics} (see Fig. \ref{v60_human}), and a human model. The complex and full-order hybrid dynamical model has $60$ state variables and $20$ control inputs together with $16$ continuous-time domains to describe a trotting gait of the robot and a bipedal gait of the human. The performance of the closed-loop hybrid system in the presence of a discrete set of obstacles around the complex gait is investigated.

%%%%%%%%%%%%%%%%%%%%%%%%%%%%%%%%%%%%%%%%%%%%%%%%%%%%%%%%%%%%%%%%%%%%%%%%%%%%%%%%

\vspace{-0.65em} %added by VK
\section{HYBRID MODELS OF LOCOMOTION}
\label{HYBRID MODELS OF LOCOMOTION}
\vspace{-0.3em}

Hybrid models of locomotion can be described by directed cycles. In this section, we will first present the hybrid models for the locomotion of each agent (i.e., robot and human). We will then address the complex hybrid model that describes the cooperative locomotion of agents.

%%%%%%%%%%%%%%%%%%%%%%%%%%%%%%%%%%%%%%%%%%%%%%%%%%%%%%%%%%%%%%%%%%%%%%%%%%%%%%%%

\vspace{-1em}
\subsection{Directed Cycles}
\label{Directed Cycles}
\vspace{-0.3em}

Throughout this paper, we shall consider \textit{multi-domain} hybrid models described by the following tuple \cite{Hamed_Ma_Ames_Vision60}
\begin{equation}\label{open_loop_hybrid_model}
\Sigma\left(\mathcal{G},\mathcal{D},\mathcal{S},\Delta,FG\right),
\end{equation}
where $\mathcal{G}:=(\mathcal{V},\mathcal{E})$ represents a \textit{direct cycle} (i.e., graph) for the studied locomotion pattern (see Fig. \ref{Illustration_control_scheme}a). In our formulation, the vertices $\mathcal{V}$ denote the continuous-time dynamics of legged locomotion, referred to as \textit{domains} or \textit{phases}. The edges $\mathcal{E}\subseteq\mathcal{V}\times\mathcal{V}$ represent the discrete-time transitions among continuous-time dynamics arising from changes in physical constraints (e.g., a new contact point is added to the set of existing contact points with the ground or an existing contact point leaves the ground). For every $v_{i},v_{j}\in\mathcal{V}$, $e=(v_{i}\rightarrow{}v_{j})\in\mathcal{E}$ if $v_{i}$ and $v_{j}$ are adjacent on $\mathcal{G}$. The state variables and control inputs of the hybrid system are shown by $x\in\mathcal{X}$ and $u\in\mathcal{U}$, respectively. The set of state manifolds and set of admissible controls are then denoted by $\mathcal{X}:=\{\mathcal{X}_{v}\}_{v\in\mathcal{V}}$ and $\mathcal{U}:=\{\mathcal{U}_{v}\}_{v\in\mathcal{V}}$, in which $\mathcal{X}_{v}$ and $\mathcal{U}_{v}$ are the state space and admissible controls for the domain $v\in\mathcal{V}$. The set of domains of admissibility are further represented by $\mathcal{D}:=\{\mathcal{D}_{v}\}_{v\in\mathcal{V}}$, where $\mathcal{D}_{v}\subseteq\mathcal{X}_{v}\times\mathcal{U}_{v}$ denotes the set of all points $(x,u)$ on which the unilateral constraints and friction cone conditions are satisfied (i.e., legs are above the walking surface and the foot slippage does not occur). The evolution of the hybrid system during the continuous-time domain $v\in\mathcal{V}$ is described by an ordinary differential equation (ODE) arising from the Euler-Lagrange equations as $\dot{x}=f_{v}(x)+g_{v}(x)\,u$ for all $(x,u)\in\mathcal{D}_{v}$. In addition, $FG:=\{(f_{v},g_{v})\}_{v\in\mathcal{V}}$ represents the set of control systems on $\mathcal{D}$. In order to simplify the presentation, we define the \textit{next domain function} as $\mu:\mathcal{V}\rightarrow\mathcal{V}$ by $\mu(v_{i})=v_{j}$ if $e=(v_{i}\rightarrow{}v_{j})\in\mathcal{E}$. The evolution of the hybrid system during the discrete-time transition $e\in\mathcal{E}$ is further described by the instantaneous mapping $x^{+}=\Delta_{e}(x^{-})$, where $x^{-}(t):=\lim_{\tau\nearrow t}x(\tau)$ and $x^{+}(t):=\lim_{\tau\searrow t}x(\tau)$ represent the state of the system right before and after the discrete transition, respectively. $\Delta:=\{\Delta_{e}\}_{e\in\mathcal{E}}$ denotes the set of discrete-time dynamics. The guards of the hybrid system are finally given by $\mathcal{S}:=\{\mathcal{S}_{e}\}_{e\in\mathcal{E}}$, on which the state trajectories undergo an abrupt change according to the discrete-time dynamics $\Delta_{e}$ when the state and control trajectories $(x,u)$ hit the surface $\mathcal{S}_{e}$ in $\mathcal{D}$.

% Here, the vector filed $f_{v}$ and columns of $g_{v}$ are assumed to be smooth functions.

%%%%%%%%%%%%%%%%%%%%%%%%%%%%%%%%%%%%%%%%%%%%%%%%%%%%%%%%%%%%%%%%%%%%%%%%%%%%%%%%

\vspace{-1em}
\subsection{Continuous-Time Dynamics}
\label{Continuous-Time Dynamics}
\vspace{-0.3em}

In this section, we consider the continuous-time dynamics for each agent. We assume that $q\in\mathcal{Q}\subset\Real^{n}$ denotes the configuration variables for the robot and/or human. The configuration space is further represented by $\mathcal{Q}$. The state vector is taken as $x:=\textrm{col}(q,\dot{q})\in\textrm{T}\mathcal{Q}$, where $\textrm{T}\mathcal{Q}$ denotes the tangent bundle of $\mathcal{Q}$. We remark that the Vision 60 robot has $n=18$ DOFs. For the human model, we make use of an $n=12$ DOF tree structure with a torso and two identical legs consisting of a femur and tibia links. The control inputs $u\in\mathcal{U}\subset\Real^{m}$ are finally taken as torques at the joint levels (i.e., $m=12$ for the dog robot and $m=6$ for the human model) (see Section \ref{NUMERICAL SIMULATIONS AND RESULTS} for further details on the models).

It is supposed that $\eta_{v}(q)\equiv0$ represents the holonomic constraints during the domain $v\in\mathcal{V}$ arising form the contact conditions between the leg ends and the ground. The equations of motion during the continuous-time domain $v$ are then described by the Euler-Lagrange equations and principle of virtual work as follows
\begin{alignat}{6}
&D(q)\,\ddot{q}+C\left(q,\dot{q}\right)\dot{q}+G(q)&&=B\,u+J_{v}^\top(q)\,\lambda\nonumber\\
&J_{v}(q)\,\ddot{q}+\frac{\partial}{\partial q}\left(J_{v}(q)\,\dot{q}\right)\dot{q}&&=0,\label{EL_equations}
\end{alignat}
where $D(q)\in\Real^{n\times{}n}$ denotes the positive definite mass-inertia matrix, $C(q,\dot{q})\,\dot{q}+G(q)\in\Real^{n}$ represents the Coriolis, centrifugal, and gravitational terms, $B\in\Real^{n\times{m}}$ denotes the input distribution matrix, $\lambda$ represents the Lagrange multipliers (i.e., ground reaction forces), and $J_{v}(q):=\frac{\partial \eta_{v}}{\partial q}(q)$ is the contact Jacobian matrix. If $J_{v}$ has full rank, one can eliminate the Lagrange multipliers to express \eqref{EL_equations} as
\begin{equation}\label{EL_equations_02}
D(q)\,\ddot{q}+H_{v}\left(q,\dot{q}\right)=T_{v}(q)\,u,
\end{equation}
in which $H_{v}:=\textrm{proj}_{v}\,H+J_{v}^\top\,(J_{v}\,D^{-1}\,J_{v}^\top)^{-1}\frac{\partial}{\partial q}(J_{v}\,\dot{q})\dot{q}$, $H=C(q,\dot{q})\,\dot{q}+G(q)$, $T_{v}:=\textrm{proj}_{v}\,B$, and $\textrm{proj}_{v}:=I-J_{v}^\top(J_{v}\,D^{-1}\,J_{v}^\top)^{-1}J_{v}\,D^{-1}$. We remark that \eqref{EL_equations_02} can be expressed as an input-affine system, i.e., $\dot{x}=f_{v}(x)+g_{v}(x)\,u$.

%  out which $6$ DOFs describe the absolute position/orientation of the robot with respect to an inertial world frame and the remaining $12$ DOFs are actuated and internal joint variables that represent the shape of the robot
% For both structures, we further assume $2$-DOF hip joints (i.e., roll and pitch) with point foot models.
% $6$ DOFs for the absolute position/orientation and $6$ DOFs for the actuated and internal joint variables

%%%%%%%%%%%%%%%%%%%%%%%%%%%%%%%%%%%%%%%%%%%%%%%%%%%%%%%%%%%%%%%%%%%%%%%%%%%%%%%%

\vspace{-1em}
\subsection{Discrete-Time Dynamics}
\label{Discrete-Time Dynamics}
\vspace{-0.3em}

If a new contact point is added to the existing set of contact points with the ground, we employ a rigid impact model \cite{Hurmuzlu_Impact} to describe the abrupt changes in the velocity coordinates according to the impact. In particular, if $\delta\lambda$ represents the intensity of the impulsive ground reaction force on the contacting points, integrating \eqref{EL_equations} over the infinitesimal period of the impact (i.e., $[t^{-},t^{+}]$) yields
\begin{equation}\label{impact}
D(q)\,\dot{q}^{+}-D(q)\,\dot{q}^{-}=J_{\mu(v)}^\top\,\delta\lambda,\,\,\,J_{\mu(v)}(q)\,\dot{q}^{+}=0,
\end{equation}
in which $\dot{q}^{-}$ and $\dot{q}^{+}$ represent the generalized velocity vector right before and after the impact, respectively. From the continuity of position, we assume $q^{+}=q^{-}$ and then from \eqref{impact}, one can solve for $\dot{q}^{+}$ and $\delta\lambda$ in terms of $(q^{-},\dot{q}^{-})$ to have a closed-form expression as $x^{+}=\Delta_{e}(x^{-})$. Furthermore, if the leg leaves the ground, we take $\Delta_{e}$ as the identity map to preserve the continuity of position and velocity coordinates over the corresponding discrete transition.

%%%%%%%%%%%%%%%%%%%%%%%%%%%%%%%%%%%%%%%%%%%%%%%%%%%%%%%%%%%%%%%%%%%%%%%%%%%%%%%%

\vspace{-1em}
\subsection{Complex Hybrid Models for Cooperative Locomotion}
\label{Complex Hybrid Model for Cooperative Locomotion}
\vspace{-0.3em}

Throughout this paper, we shall assume that there is a rigid and massless leash model that connects a point on the dog (e.g., head) to a point on the human (e.g., hand or hip) via ball (i.e., socket) joints. The leash will further be assumed to be actuated to control its length and orientation so that the human can follow the dog in a safe manner. This will be clarified with more details in Section \ref{Leash Baseline Controller}. The state and control inputs for the robotic dog and human are shown by $x^{i}:=\textrm{col}(q^{i},\dot{q}^{i})$ and $u^{i}$, respectively, for $i\in\{d,h\}$, where the superscripts ``$d$'' and ``$h$'' stand for the dog and human.

\noindent\textbf{Complex Graph:} The complex hybrid model that describes the cooperative locomotion of the robot and human will have a complex graph that is taken as the strong product of graphs $\mathcal{G}^{d}=(\mathcal{V}^{d},\mathcal{E}^{d})$ and $\mathcal{G}^{h}=(\mathcal{V}^{h},\mathcal{E}^{h})$. The strong product is denoted by $\mathcal{G}^{c}:=\mathcal{G}^{d}\boxtimes\mathcal{G}^{h}$
that has the vertex set $\mathcal{V}^{c}:=\mathcal{V}^{d}\times\mathcal{V}^{h}$, and any two vertices $(v,w)$ and $(v',w')$ in $\mathcal{V}^{c}$ are adjacent if and only if 1) $v=v'$ and $(w\rightarrow{}w')$ is an edge in $\mathcal{E}^{h}$, or 2) $(v\rightarrow{}v')$ is an edge in $\mathcal{E}^{d}$ and $w=w'$, or 3) $(v\rightarrow{}v')$ is an edge in $\mathcal{E}^{d}$ and $(w\rightarrow{}w')$ is an edge in $\mathcal{E}^{h}$. In our notation, the superscript ``$c$'' represents the complex model. The augmented state and control inputs are further denoted by $x^{c}:=\textrm{col}(x^{d},x^{h})$ and $u^{c}:=\textrm{col}(u^{d},u^{h})$, respectively.

\noindent\textbf{Complex Continuous-Time Dynamics:} For every vertex $(v,w)\in\mathcal{V}^{c}$, the evolution of the composite mechanical system, consisting of the robot and human, can be described by the following nonlinear and \textit{coupled} dynamics
\begin{alignat}{8}
&&D^{d}\left(q^{d}\right)\ddot{q}^{d}&+H^{d}_{v}\left(q^{d},\dot{q}^{d}\right)&=T_{v}^{d}\left(q^{d}\right)u^{d}&-J_{\textrm{head}}^{d\top}\left(q^{d}\right)F\nonumber\\
&&D^{h}\left(q^{h}\right)\ddot{q}^{h}&+H^{h}_{w}\left(q^{h},\dot{q}^{h}\right)&=T_{w}^{h}\left(q^{h}\right)u^{h}&+J_{\textrm{hand}}^{h\top}\left(q^{h}\right)F,\label{coupled_dyn}
\end{alignat}
in which $J_{\textrm{head}}^{d}(q^{d})$ and $J_{\textrm{hand}}^{h}(q^{h})$ denote the Jacobian matrices for the end points of the leash at the dog and human sides, respectively, and $F\in\Real^{3}$ represents the force applied by the leash to the human hand. We remark that the superscripts ``$d$'' and ``$h$'' represent the dynamic and kinematic terms for the dog and human models, respectively.

\noindent\textbf{Complex Discrete-Time Dynamics:} Since the leash model is assumed to be massless and cannot employ impulsive forces, the evolution of the composite mechanical system over the discrete transition $(v,w)\rightarrow(v',w')$ can be described by the following nonlinear and \textit{decoupled} mappings
\begin{equation}\label{complex_impact}
x^{d+}=\Delta_{v\rightarrow{}v'}^{d}\left(x^{d-}\right),\quad
x^{h+}=\Delta_{w\rightarrow{}w'}^{h}\left(x^{h-}\right).
\end{equation}
We remark that if $v=v'$ (resp. $w=w'$) in \eqref{complex_impact}, the mapping $\Delta_{v\rightarrow{}v'}$ (resp. $\Delta_{w\rightarrow{}w'}$) is simply taken as the identity.

\begin{remark}
In this paper, we shall consider a trotting gait for Vision 60 robot with $8$ continuous-time domains (see Fig. \ref{Illustration_control_scheme}a for more details). The graph for the bipedal gait of the human model also has $2$ continuous-time domains that represent the right and left stance phases. Consequently, the complex hybrid model of locomotion would have $8\times2=16$ continuous-time domains for which there are $2\times(n^{d}+n^{h})=2\times(18+12)=60$ state variables and $m^{d}+m^{h}+m^{l}=12+6+2=20$ control inputs. Here, $m^{l}=2$ represents the actuator numbers for the leash (see Section \ref{Leash Baseline Controller}).
\end{remark}

%%%%%%%%%%%%%%%%%%%%%%%%%%%%%%%%%%%%%%%%%%%%%%%%%%%%%%%%%%%%%%%%%%%%%%%%%%%%%%%%

\vspace{-1em}
\section{HIERARCHICAL CONTROL STRATEGY}
\label{HIERARCHICAL CONTROL STRATEGY}
\vspace{-0.5em}

In order to have stable and safe cooperative locomotion for the robot and human, we will present a two-level control strategy for the robotic dog and leash (see Fig. \ref{Illustration_control_scheme}b). Since the mathematical models for the local controller of the human part are not known, we shall assume a nonlinear local controller for the human, but will \textit{not} change that controller to address unforeseen events and obstacle avoidance. We will instead focus on the dog and leash hierarchical control strategy to ensure stability and safety. At the higher level of the control strategy, we will employ a local nonlinear controller for the robot that has access to its own state variables as well as the force employed by the leash (i.e., force measurement). This controller will be referred to as the \textit{robot baseline controller}. The objective is to asymptotically derive some outputs to zero that encode the locomotion patterns for the guide robot (e.g., trot, amble, walk, or gallop gaits at desired speeds). The baseline controller for the dog will be designed via HZD and virtual constraints approach in Section \ref{LOCAL VIRTUAL CONSTRAINT CONTROLLERS}. This controller exponentially stabilizes gaits for the hybrid model of the dog in the presence of the leash force. The baseline controller for the leash will be designed to ensure that 1) there is always a safe distance between the robot and human, and 2) human follows the robot (see Section \ref{Leash Baseline Controller}). At the lower level, we will solve a real-time QP optimization to ensure safety and obstacle avoidance. In particular, the QP optimization modifies the baseline controllers for the robot as well as the leash to keep the dog and human in a safe distance from obstacles. The QP framework will be set up based on CBFs in Section \ref{SAFETY CRITICAL CONTROL AND CONVEX QP OPTIMIZATION}.

%%%%%%%%%%%%%%%%%%%%%%%%%%%%%%%%%%%%%%%%%%%%%%%%%%%%%%%%%%%%%%%%%%%%%%%%%%%%%%%%

\vspace{-1em}
\subsection{Local Baseline Controllers for the Agents}
\vspace{-0.3em}

In this section, we consider the robot and human as two multi-body ``agents'' specified by the superscript $i\in\{d,h\}$.

\begin{definition}[Local Baseline Controllers]
We suppose that there are local and smooth feedback laws $\Gamma^{i}(x^{i},F):=\{\Gamma_{v}^{i}(x^{i},F)\}_{v\in\mathcal{V}^{i}}$ for the agent $i\in\{d,h\}$ to have stable locomotion patterns. In our notation, $\Gamma_{v}^{i}(x^{i},F)$ is a local and nonlinear feedback controller, referred to as \textit{baseline controller}, that is employed during the continuous-time $v\in\mathcal{V}^{i}$ and assumed to have access to the state variables of the agent $i$ (i.e., $x^{i}$) as well as the force $F$.
\end{definition}

\begin{assumption}[Transvsersal Stable Periodic Orbits]
\label{Transvsersal Stable Periodic Orbits}
By employing the local baseline controllers for the agent $i\in\{d,h\}$ in the unleashed case (i.e., $F\equiv0$), we assume that there is a period-one orbit (i.e., gait) for the closed-loop hybrid model $\Sigma^{i}$, denoted by $\mathcal{O}^{i}_{\textrm{ul}}$, that is transversal to the guards $\mathcal{S}^{i}$. In our notation, the subscript ``$\textrm{ul}$'' stands for the unleashed gait. The orbit $\mathcal{O}^{i}_{\textrm{ul}}$ is further supposed to be locally exponentially stable.
\end{assumption}

For future purposes, the evolution of the state variables $x^{i}$ on the unleashed orbit $\mathcal{O}^{i}_{\textrm{ul}}$ is represented by $x_{\star}^{i}(t)$ for $t\geq0$. The orbit $\mathcal{O}^{i}_{\textrm{ul}}$ can then be expressed as
\begin{equation}
\mathcal{O}^{i}_{\textrm{ul}}:=\left\{x^{i}=x_{\star}^{i}(t)\,|\,0\leq{t}<T^{i}\right\},
\end{equation}
in which $T^{i}>0$ denotes the minimal period of $x_{\star}^{i}(t)$. Section \ref{LOCAL VIRTUAL CONSTRAINT CONTROLLERS} will present a class of HZD-based local baseline controllers to exponentially stabilize the periodic gaits $\mathcal{O}^{i}_{\textrm{ul}}$.

\begin{assumption}[Common Multiples of Gait Periods]
\label{Common Multiples of Gait Periods}
We assume that there are common multiples for the periods of the dog and human unleashed gaits. More specifically, there are positive integers $N^{d}$ and $N^{h}$ such that $N^{d}\,T^{d}=N^{h}\,T^{h}$. For future purposes, we denote the minimum of these values by $N^{d}_{\min}$ and $N^{h}_{\min}$.
\end{assumption}

%%%%%%%%%%%%%%%%%%%%%%%%%%%%%%%%%%%%%%%%%%%%%%%%%%%%%%%%%%%%%%%%%%%%%%%%%%%%%%%%

\vspace{-0.5em}
\subsection{Leash Baseline Controller}
\label{Leash Baseline Controller}
\vspace{-0.3em}

As mentioned previously, the leash structure is assumed to be rigid. We further suppose that its length and orientation can be independently controlled by two linear and rotational actuators. To make this notion more precise, let us denote the Cartesian coordinates of the dog head and the human hand by $p^{d}_{\textrm{head}}(q^{d})\in\Real^{3}$ and $p^{h}_{\textrm{hand}}(q^{h})\in\Real^{3}$, respectively. Next, consider the vector connecting $p^{h}_{\textrm{hand}}(q^{h})$ to $p^{d}_{\textrm{head}}(q^{d})$. The representation of this vector in the cylindrical coordinates can be given by $(r,\theta,z)$. We assume that there are sensors for the leash structure to measure $(r,\theta)$. The objective here is to design a local force feedback controller for the leash that has access to $(r,\theta)$ to keep the human in a safe distance from the robot dog while regulating the angle $\theta$. In particular, we are interested in (i) having $r\in[r_{\min},r_{\max}]$ for some $0<r_{\min}<r_{\max}$ and (ii) imposing $\theta\rightarrow0$. This controller is referred to as the \textit{leash baseline controller}. One possible way to design such a controller is to decompose the force $F$ into $(F_{r},F_{\theta},F_{z})$, in which $F_{r}(r,\dot{r})$ is the longitudinal force designed to be sufficiently differentiable while being zero over the safe zone $[r_{\min},r_{\max}].$ Moreover, $F_{\theta}(\theta,\dot{\theta})$ is a torsional force that can be taken as a simple PD controller to regulate $\theta$. For the purpose of this paper, $F_{z}$ is assumed to be zero. For future purposes, the leash baseline controller will be represented by $F_{b}(r,\dot{r},\theta,\dot{\theta},\kappa)\in\Real^{3}$, where the subscript ``$b$'' stands for the baseline control and $\kappa$ represents some adjustable controller parameters, e.g., PD gains.

\begin{assumption}
\label{leash_assumption}
We assume that $F_{b}$ is sufficiently differentiable with respect to its arguments $(r,\dot{r},\theta,\dot{\theta},\kappa)$. Furthermore, for $\kappa=0$, $F_{b}(r,\dot{r},\theta,\dot{\theta},\kappa)\equiv0$.
\end{assumption}

\begin{remark}
Since the longitudinal force $F_{r}(r,\dot{r})$ is assumed to be zero over the safe zone $[r_{\min},r_{\max}]$, $F_{r}$ would have a deadzone structure over $[r_{\min},r_{\max}]$. Assumption \ref{leash_assumption} ensures that $F_{r}$ is designed to be differentiable at the corners $r_{\min}$ and $r_{\max}$ such that the stability analysis can be carried out via the Poincar\'e return map in Theorem \ref{Stability of Complex Gaits with Leash}.
\end{remark}

%%%%%%%%%%%%%%%%%%%%%%%%%%%%%%%%%%%%%%%%%%%%%%%%%%%%%%%%%%%%%%%%%%%%%%%%%%%%%%%%

\vspace{-1em}
\subsection{Stability Analysis of Complex Gaits}
\vspace{-0.3em}

This section addresses the existence and stability of periodic orbits for the cooperative locomotion of the robot and human in the presence of leash. We make use of the Poincar\'e sections analysis and present the following theorem.

\begin{theorem}[Stability of Complex Gaits with Leash]
\label{Stability of Complex Gaits with Leash}
Under Assumptions \ref{Transvsersal Stable Periodic Orbits}-\ref{leash_assumption}, there is an open neighborhood of $0$, denoted by $\mathcal{N}(0)$, such that for all gain values $\kappa\in\mathcal{N}(0)$, there is an exponentially stable complex gait for the leashed closed-loop hybrid system $\Sigma^{c}$.
\end{theorem}

\begin{proof}
From Assumptions \ref{Transvsersal Stable Periodic Orbits} and \ref{Common Multiples of Gait Periods}, the following augmented orbit
\begin{equation}
\mathcal{O}^{c}_{\textrm{ul}}:=\left\{x^{c}=\textrm{col}(x_{\star}^{d}(t),x_{\star}^{h}(t))\,|\,0\leq{t}<N^{d}_{\min}\,T^{d}\right\}
\end{equation}
is indeed a periodic orbit for the complex and unleashed hybrid system $\Sigma^{c}$. We next choose a Poincar\'e section transversal to this orbit, denoted by $\mathscr{S}$, and consider a Poincar\'e return map for $\Sigma^{c}$ from $\mathscr{S}$ back to $\mathscr{S}$ as $P^{c}(x^{c},\kappa)$. According to the construction procedure, there is a fixed point for the Poincar\'e map that corresponds to $\mathcal{O}^{c}_{\textrm{ul}}$, that is $P^{c}(x_{\star,\textrm{ul}}^{c},0)=x_{\star,\textrm{ul}}^{c}$, in which  $x_{\star,\textrm{ul}}^{c}$ represents the fixed point. We next consider the algebraic equation $E(x^{c},\kappa):=P^{c}\left(x^{c},\kappa\right)-x^{c}=0$. Since $\mathcal{O}_{\textrm{ul}}^{c}$ is exponentially stable for the unleashed complex system, the Jacobian matrix $\frac{\partial E}{\partial x^{c}}(x_{\star,\textrm{ul}}^{c},0)=\frac{\partial P^{c}}{\partial x^{c}}(x_{\star,\textrm{ul}}^{c},0)-I$ is nonsingular. Hence, from the Implicit Function Theorem, there exists $\mathcal{N}(0)$ such that for all $\kappa\in\mathcal{N}(0)$, there is a fixed point for $P^{c}(x^{c},\kappa)$. Moreover, since the elements and eigenvalues of the Jacobian matrix $\frac{\partial P^{c}}{\partial x^{c}}(x^{c},\kappa)$ continuously depend on $\kappa$, one can choose $\mathcal{N}(0)$ sufficiently small such that the eigenvalues of the Jacobian matrix remain inside the unit circle. This completes the proof of exponential stability for leashed locomotion.
\end{proof}

%%%%%%%%%%%%%%%%%%%%%%%%%%%%%%%%%%%%%%%%%%%%%%%%%%%%%%%%%%%%%%%%%%%%%%%%%%%%%%%%

\vspace{-1em}
\section{LOCAL VIRTUAL CONSTRAINT CONTROLLERS WITH FORCE FEEDBACK}
\label{LOCAL VIRTUAL CONSTRAINT CONTROLLERS}
\vspace{-0.3em}

The objective of this section is to design the local baseline controller for the robotic dog. The controller is designed based on virtual constraints approach \cite{Grizzle_Asymptotically_Stable_Walking_IEEE_TAC,Westervelt_Grizzle_Koditschek_HZD_IEEE_TRO} to ensure exponential stability of the gait for the unleashed case. Virtual constraints are defined as kinematic constraints (i.e., outputs) that encode the locomotion pattern. They are imposed through the action of the baseline controllers. The idea is to coordinate the motion of the links within domains. We make use of relative degree one and relative degree two virtual constraints (i.e., outputs). In particular, during the continuous-time domain $v\in\mathcal{V}^{d}$, we consider the following outputs to be regulated
\begin{equation}\label{virtual_constraint_outputs}
y_{v}^{d}\left(x^{d}\right):=\begin{bmatrix}
y_{1v}^{d}(q^{d},\dot{q}^{d})\\
y_{2v}^{d}(q^{d})
\end{bmatrix},
\end{equation}
in which $y_{1v}^{d}(q^{d},\dot{q}^{d})$ represents relative degree one nonholonomic outputs for velocity regulation and  and $y_{2v}^{d}(q^{d})$ denotes relative degree two holonomic outputs for position tracking. Using the nonlinear dynamics \eqref{coupled_dyn} and standard input-output linearization \cite{Isidori_Book}, one can obtain
\begin{equation}\label{output_dyn_0}
\begin{bmatrix}
\dot{y}_{1v}^{d}\\
\ddot{y}_{2v}^{d}
\end{bmatrix}=A_{v}^{d}\left(x^{d}\right)u^{d}+b_{v}^{d}\left(x^{d},F\right),
\end{equation}
where $A_{v}^{d}(x)$ is a decoupling matrix and $b_{v}^{d}$ consists of Lie derivatives (see \cite{Hamed_Ma_Ames_Vision60} for more details). Furthermore, we would like to solve for $u^{d}$ that results in the following output dynamics
\begin{equation}\label{output_dyn}
\begin{bmatrix}
\dot{y}_{1v}^{d}\\
\ddot{y}_{2v}^{d}
\end{bmatrix}=-\ell_{v}(x^{d}):=-\begin{bmatrix}
K_{P}\,y_{1d}^{v}\\
K_{D}\,\dot{y}_{2v}^{d}+K_{P}\,y_{2v}^{d}
\end{bmatrix}
\end{equation}
with $K_{P}$ and $K_{D}$ being positive PD gains. The local baseline controller for the dog is finally chosen as
\begin{equation}\label{HZD_controllers}
\Gamma_{v}^{d}\left(x^{d},F\right):=-A_{v}^{d\top}\left(A_{v}^{d}\,A_{v}^{d\top}\right)^{-1}\left(b_{v}^{d}+\ell_{v}\right)
\end{equation}
that 1) requires local state and force measurement and 2) exponentially stabilizes the equilibrium point $(y_{1v}^{d},y_{2v}^{d},\dot{y}_{2v}^{d})=(0,0,0)$ for the output dynamics \eqref{output_dyn} in the presence of the external force $F$, i.e., $\lim_{t\rightarrow\infty}y_{v}^{d}(t)=0$.

\begin{remark}[Proper Selection of Virtual Constraints]
The periodic orbit $\mathcal{O}^{d}_{\textrm{ul}}$ can be designed in an offline manner through direct collocation based trajectory optimization techniques \cite{FROST,Ames_DURUS_TRO}. For a given periodic gait $\mathcal{O}_{\textrm{ul}}^{d}$, the output functions $y_{v}^{d}$ in \eqref{virtual_constraint_outputs} are chosen to vanish on the desired gait $\mathcal{O}_{\textrm{ul}}^{d}$. We have observed that the stability of gaits in the virtual constraint approach depends on the proper selection of the output functions $y_{v}^{d}$ to be regulated \cite{Hamed_Buss_Grizzle_BMI_IJRR}. Our previous work \cite{Hamed_Buss_Grizzle_BMI_IJRR,Hamed_Gregg_decentralized_control_IEEE_CST} has developed a recursive algorithm, based on semidefinite programming, to systematically design output functions for which the gaits are exponentially stable for the corresponding hybrid dynamics. The algorithm is offline and assumes a finite-dimensional parameterization of the output functions to be determined. Then it translates the exponential stabilization problem into a recursive optimization problem that is set up based on linear and bilinear matrix inequalities. Sufficient conditions for the convergence of the algorithm to a set of stabilizing parameters have been addressed in \cite{Hamed_Gregg_decentralized_control_IEEE_CST,Hamed_Gregg_Ames_ACC}.
\end{remark}

\begin{remark}
Nonlinear local controllers for the human model are not know. However, for the purpose of this paper, we assume virtual constraint-based controllers, analogous to \eqref{HZD_controllers}, for the human model. Furthermore, evidence suggests that the phase-dependent models can reasonably predict human joint behavior across perturbations \cite{Villarreal:TRO}.
\end{remark}

%%%%%%%%%%%%%%%%%%%%%%%%%%%%%%%%%%%%%%%%%%%%%%%%%%%%%%%%%%%%%%%%%%%%%%%%%%%%%%%%

\vspace{-0.8em}
\section{SAFETY-CRITICAL CONTROL AND QP OPTIMIZATION}
\label{SAFETY CRITICAL CONTROL AND CONVEX QP OPTIMIZATION}
\vspace{-0.5em}

This section aims to develop low-level safety-critical control algorithms that ensure obstacle avoidance while implementing the baseline controllers for the agents and leash in the  hierarchical control structure. We will address safety critical conditions through set invariance and CBFs. In particular, a system being safe is commonly defined as the system never leaving the safety set \cite{ames2017cbf,gurriet2018online,nguyen20163d}. For low-dimensional dynamical systems, analytical control strategies can be derived. However, finding such a control policy for high-DOF and complex models of cooperative legged locomotion of guide dogs and humans is a challenge. To tackle this problem, we make use of a real-time QP formulation to address safety specifications represented by CBFs \cite{ames2017cbf}. To present the main idea, let us consider a discrete set of static and point obstacles $\mathscr{P}^{o}_{\alpha}$ for $\alpha\in\mathcal{I}^{o}$ whose Cartesian coordinates in the $xy$-planes are given by $r_{\alpha}^{o}:=\textrm{col}(x_{\alpha}^{o},y_{\alpha}^{o})$. Next we assume a set of critical points on the robot and human that are supposed to be in a safe distance from these obstacles. These points are denoted by $\mathscr{P}^{d}_{\beta}$ and $\mathscr{P}^{h}_{\gamma}$ for the dog and human, respectively, for some $\beta\in\mathcal{I}^{d}$ and $\gamma\in\mathcal{I}^{h}$. One typical example includes the hip points of the robot and human models. The Cartesian coordinates of $\mathscr{P}^{d}_{\beta}$ and $\mathscr{P}^{h}_{\gamma}$ in the $xy$-plane are further denoted by $r_{\beta}^{d}(q^{d})\in\Real^{2}$ and $r_{\gamma}^{h}(q^{h})\in\Real^{2}$. We formulate the safety set as
\begin{alignat}{6}
&&\mathcal{C}:=\big\{x^{c}=\textrm{col}\left(x^{d},x^{h}\right)&\,|\,h_{\beta,\alpha}^{d}\left(q^{d}\right)\geq0,\, h_{\gamma,\alpha}^{h}\left(q^{h}\right)\geq0,\nonumber\\
&& &\forall(\alpha,\beta,\gamma)\in\mathcal{I}^{o}\times\mathcal{I}^{d}\times\mathcal{I}^{h}\big\},
\end{alignat}
where $h_{\beta,\alpha}^{d}(q^{d}):=\|r^{d}_{\beta}(q^{d})-r_{\alpha}^{o}\|_{2}^{2}-h_{\min}^{2}$ and $h_{\gamma,\alpha}^{h}(q^{h}):=\|r^{h}_{\gamma}(q^{h})-r_{\alpha}^{o}\|_{2}^{2}-h_{\min}^{2}$ for some safety distance $h_{\min}>0$. In addition, $\|.\|_{2}$ denotes the Euclidean norm. The safety constraints $h_{\beta,\alpha}^{d}(q^{d})\geq0$ and $h_{\gamma,\alpha}^{h}(q^{h})\geq0$ are relative degree two. Our objective is to modify the torques for the dog robot $u^{d}$ as well as the leash force $F$ to render the safety set $\mathcal{C}$ forward invariant under the flow of the closed-loop complex model. We remark that we are \textit{not} allowed to change the human controller $u^{h}=\Gamma^{h}(x^{h},F)$ as the person can be visually impaired and cannot react properly. For this purpose, we make use of the concept of exponential CBFs (ECBFs) \cite{Sreenaath_HighDegreeCBF}. In particular, we define the ECBFs as follows
\begin{alignat}{6}
&&\mathcal{B}_{\beta,\alpha}^{d}\left(x^{d}\right)&:=\dot{h}_{\beta,\alpha}^{d}\left(x^{d}\right)&+\lambda\,h_{\beta,\alpha}^{d}\left(x^{d}\right)\label{CBFS_1}\\
&&\mathcal{B}_{\gamma,\alpha}^{h}\left(x^{h}\right)&:=\dot{h}_{\gamma,\alpha}^{h}\left(x^{h}\right)&+\lambda\,h_{\gamma,\alpha}^{h}\left(x^{h}\right)\label{CBFS_2}
\end{alignat}
for all $(\alpha,\beta,\gamma)\in\mathcal{I}^{o}\times\mathcal{I}^{d}\times\mathcal{I}^{h}=:\mathcal{I}$, where $\lambda>0$ is an adjustable parameter. The exponential CBF condition further implies that
\begin{alignat}{6}
&& \dot{\mathcal{B}}_{\beta,\alpha}^{d}\left(x^{d},u^{d},F\right)&+\omega\,\mathcal{B}_{\beta,\alpha}\left(x^{d}\right)&\geq0\label{ECBFs_01}\\
&& \dot{\mathcal{B}}_{\gamma,\alpha}^{h}\left(x^{h},F\right)&+\omega\,\mathcal{B}_{\gamma,\alpha}\left(x^{h}\right)&\geq0\label{ECBFs_02}
\end{alignat}
for all $(\alpha,\beta,\gamma)\in\mathcal{I}$ and some adjustable scalar $\omega>0$. Substituting \eqref{CBFS_1} and \eqref{CBFS_2} into \eqref{ECBFs_01} and \eqref{ECBFs_02} results in
\begin{alignat}{6}
&& \ddot{h}_{\beta,\alpha}^{d}&+(\lambda+\omega)\,\dot{h}_{\beta,\alpha}^{d}&+\lambda\,\omega\,h_{\beta,\alpha}^{d}&\geq0\label{ECBFs_01ver2}\\
&& \ddot{h}_{\gamma,\alpha}^{h}&+(\lambda+\omega)\,\dot{h}_{\gamma,\alpha}^{h}&+\lambda\,\omega\,h_{\gamma,\alpha}^{h}&\geq0\label{ECBFs_02ver2}
\end{alignat}
for every $(\alpha,\beta,\gamma)\in\mathcal{I}$. From \eqref{coupled_dyn}, we remark that \eqref{ECBFs_01ver2} and \eqref{ECBFs_02ver2} can be expressed as affine inequalities in terms of the the robot torques and leash force $(u^{d},F)$. This can be expressed as follows
\begin{alignat}{6}
&& &A_{\beta,\alpha}^{d}\left(x^{d}\right)\begin{bmatrix}
u^{d}\\
F
\end{bmatrix}&+b_{\beta,\alpha}^{d}\left(x^{d}\right)&\geq0\label{ECBFs_01ver3}\\
&& &A_{\gamma,\alpha}^{h}\left(x^{h}\right)F&+b_{\gamma,\alpha}^{h}\left(x^{h}\right)&\geq0\label{ECBFs_02ver3}
\end{alignat}
for all $(\alpha,\beta,\gamma)\in\mathcal{I}$. Next, we set up the following real-time QP to ensure safety-critical constraints while being close to the baseline controllers
\begin{alignat}{6}
\min_{(u^{d},F)} &\,\,\,\left\|u^{d}-\Gamma^{d}\left(x^{d},F\right)\right\|_{2}^{2}+\left\|F-F_{b}\left(r,\dot{r},\theta,\dot{\theta},\kappa\right)\right\|_{2}^{2}\nonumber\\
\textrm{s.t.}    &\,\,\,A_{\beta,\alpha}^{d}\left(x^{d}\right)\begin{bmatrix}
u^{d}\\
F
\end{bmatrix}+b_{\beta,\alpha}^{d}\left(x^{d}\right)\geq0,\,\,\forall(\alpha,\beta,\gamma)\in\mathcal{I}\nonumber\\
&\,\,\,A_{\gamma,\alpha}^{h}\left(x^{h}\right)F+b_{\gamma,\alpha}^{h}\left(x^{h}\right)\geq0,\quad\,\,\,\forall(\alpha,\beta,\gamma)\in\mathcal{I}\nonumber\\
&\,\,\,u_{\min}\leq{u^{d}}\leq{}u_{\max}\nonumber\\
&\,\,\,F_{\min}\leq{F}\leq{}F_{\max}\label{QP_otptimization},
\end{alignat}
where $u_{\min}$, $u_{\max}$, $F_{\min}$, and $F_{\max}$ denote the lower and upper bounds for the torques and forces. We remark that according to the construction procedure of the baseline controller in \eqref{output_dyn_0} and \eqref{HZD_controllers}, $b^{d}_{v}(x^{d},F)$ and $\Gamma^{d}_{v}(x^{d},F)$ are affine in terms of the leash force $F$ for every $v\in\mathcal{V}^{d}$. Hence, the cost function in \eqref{QP_otptimization} is indeed quadratic in terms of $(u^{d},F)$. The output of the QP framework are eventually employed as the control inputs for the robotic dog and as well as the leash.

\begin{remark}
In the QP formulation \eqref{QP_otptimization}, one would need to measure the human state variables $x^{h}$ to check for the inequality constraints \eqref{ECBFs_02ver3}. However, we do \textit{not} modify the torques for the human model. This assumption is not restrictive as one can measure the human state variable via 1) a set of wearable inertial measurement units (IMUs) and 2) asymptotic observers. In particular, our previous work \cite{Hamed_Ames_Gregg_ACC} has developed a systematic approach for asymptotic estimation of the state variables for human models via hybrid observers and IMUs. For the purpose of this paper, we hence assume that $x^{h}$ is available for the QP framework.
\end{remark}

%%%%%%%%%%%%%%%%%%%%%%%%%%%%%%%%%%%%%%%%%%%%%%%%%%%%%%%%%%%%%%%%%%%%%%%%%%%%%%%%

\vspace{-1em}
\section{NUMERICAL SIMULATIONS AND RESULTS}
\label{NUMERICAL SIMULATIONS AND RESULTS}
\vspace{-0.3em}

The objective of this section is to numerically validate the theoretical results of the paper. For this purpose, we consider a complex and full-order hybrid dynamical model that describes the cooperative locomotion of Vision 60 and a human model.

\begin{figure*}[!t]
\centering
\subfloat[\label{COM_trajectories_no_leash}]{\includegraphics[width=2.2in]{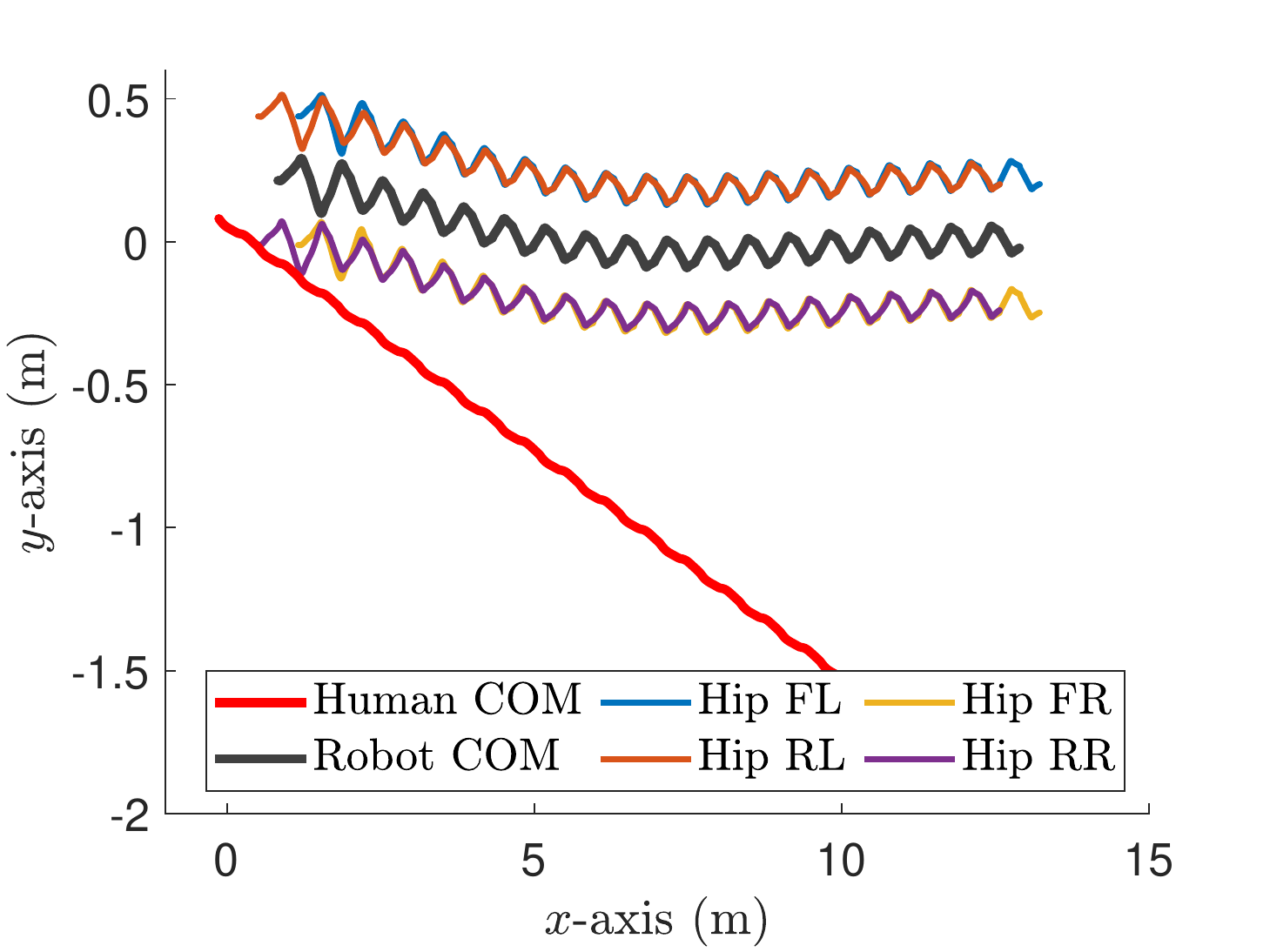}}\!\!
\subfloat[\label{COM_trajectories_with_leash}]{\includegraphics[width=2.2in]{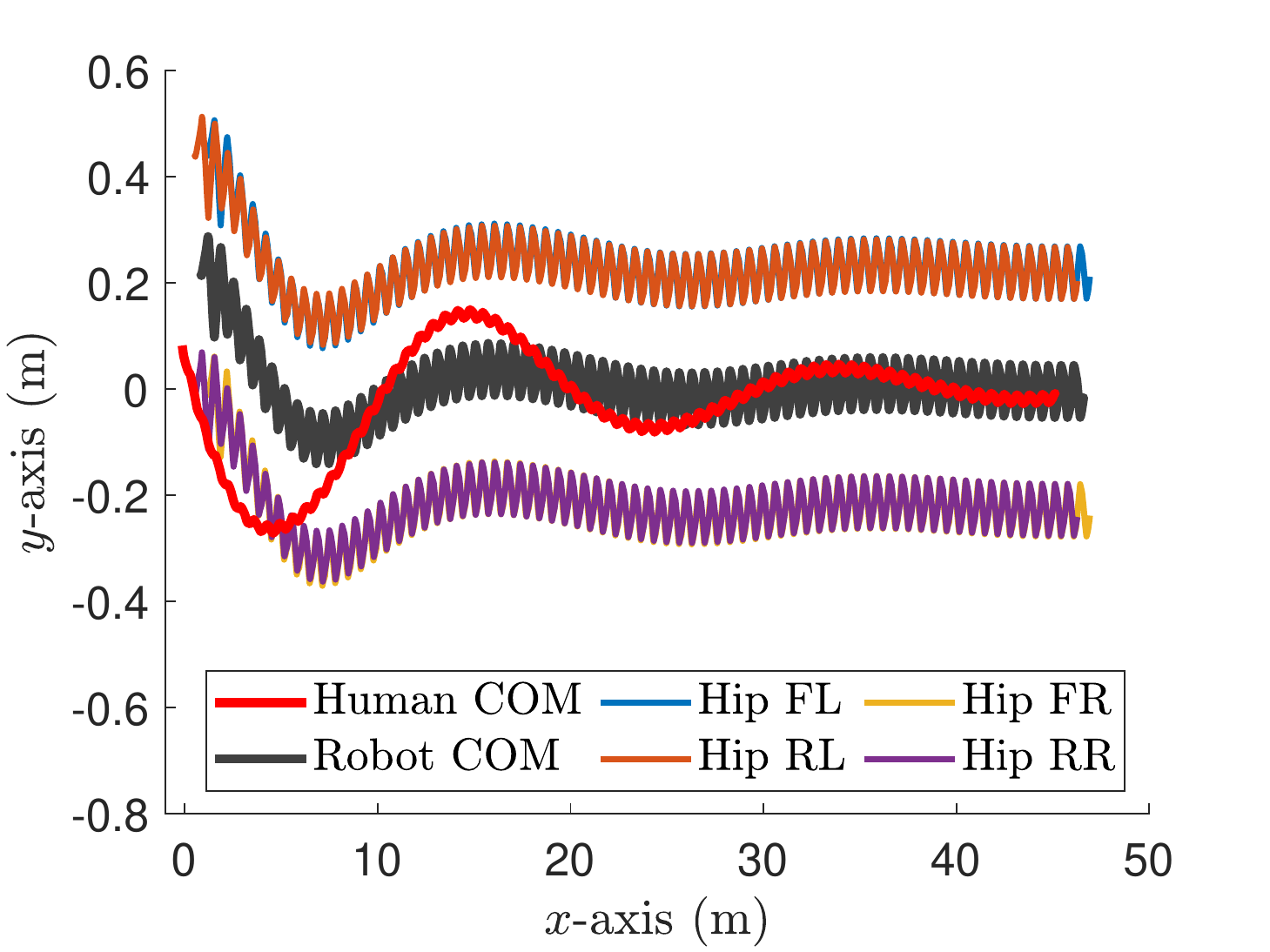}}\!\!
\subfloat[\label{COM_trajectories_point_obstacles_CBFs}]{\includegraphics[width=2.2in]{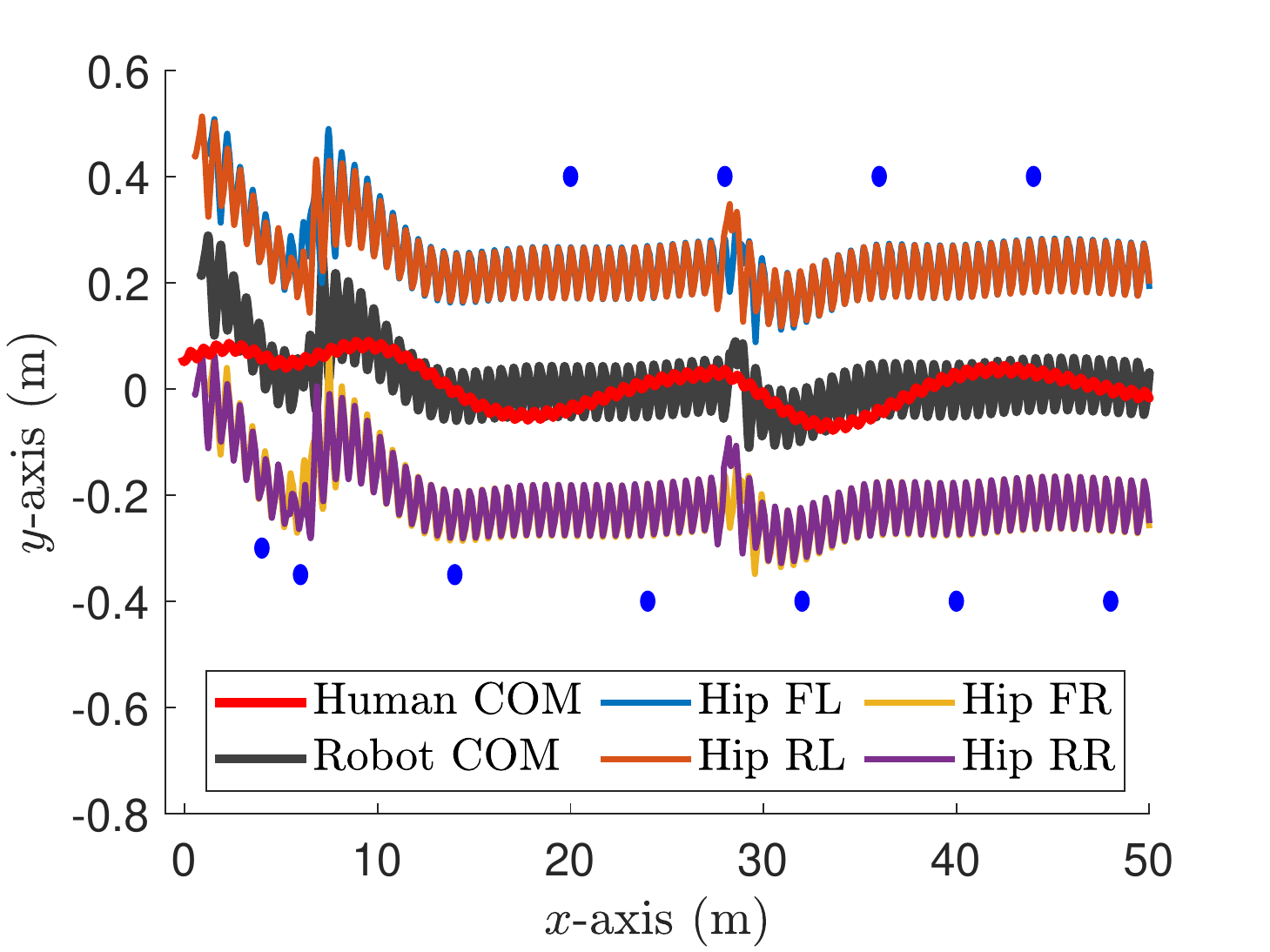}}
\vspace{-0.7em}
\caption{(a) Robot and human COM trajectories in the $xy$-plane without using the leash structure. The unleashed gait for the dog is exponentially stable (i.e, it walks along a line parallel to the $x$-axis on which the yaw angle is zero). However, the one for the human is modulo yaw stable. (b) COM trajectories using the leash structure. Here the leash and each agent have its own local baseline controllers and there is \textit{no} CBF-based QP optimization. Both the robot and human converge to a complex gait with a common speed while having yaw stability. (c) COM trajectories using the proposed hierarchical control strategy for the dog and leash structure in the presence of point obstacles. The obstacles are illustrated by the circles.}
\vspace{-1.0em}
\end{figure*}

\begin{figure*}[!t]
\centering
\subfloat[\label{Yaw_Roll_Motion_pointobstacles_CBFs:a}]{\includegraphics[width=2.2in]{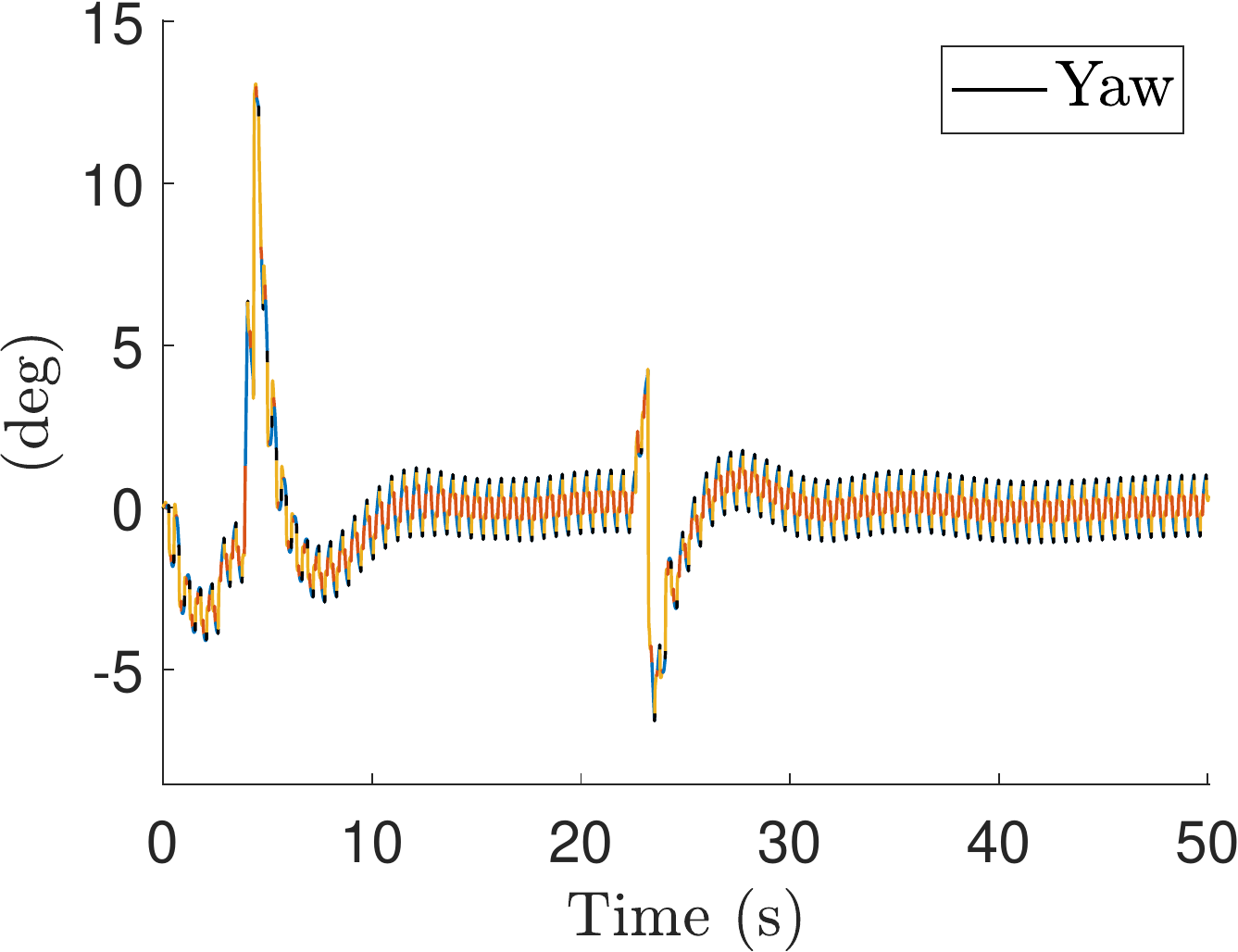}}\!\!
\subfloat[\label{Yaw_Roll_Motion_pointobstacles_CBFs:b}]{\includegraphics[width=2.2in]{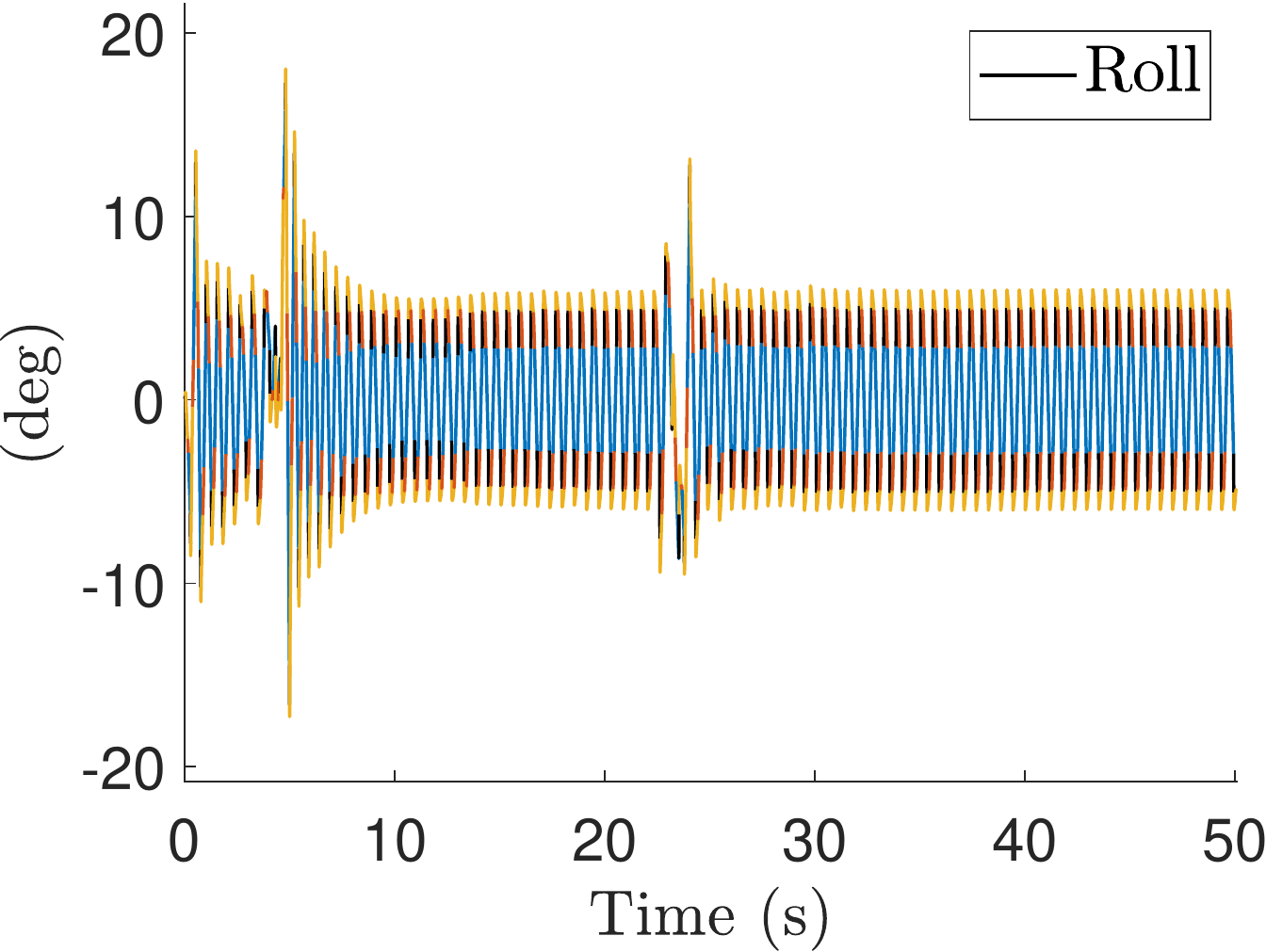}}\!\!
\subfloat[\label{COM_trajectories_point_obstacles_CBFs_moredense}]{\includegraphics[width=2.2in]{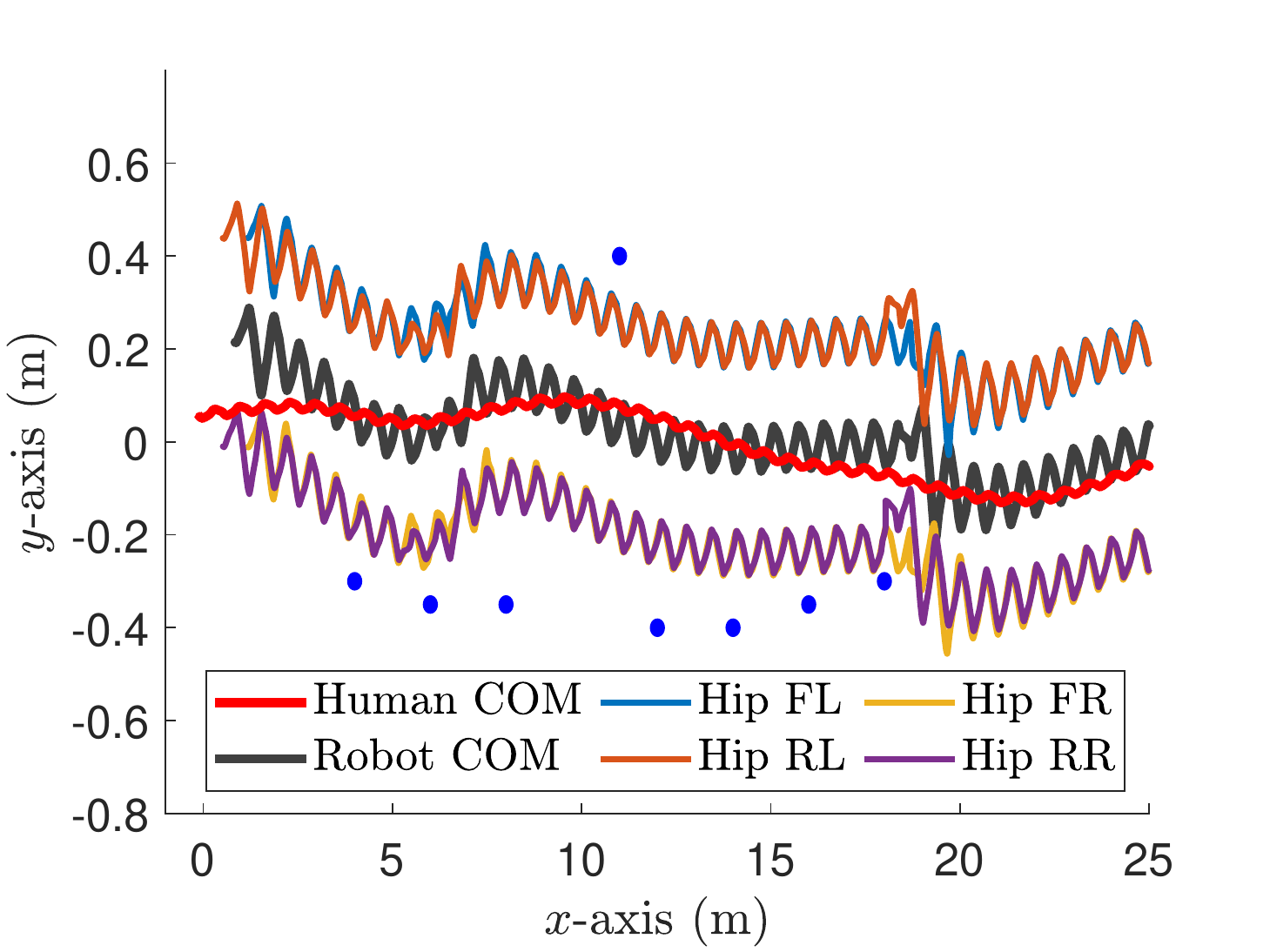}}
\vspace{-0.7em}
\caption{(a) and (b) Time profile of the yaw and roll motions for the dog robot using the proposed hierarchical control strategy in the presence of point obstacles. (c) COM trajectories in the $xy$-plane in the presence of a more-dense set of obstacles.}
\vspace{-1.0em}
\end{figure*}

\noindent\textbf{Vision 60 Robot:} Vision 60 is an autonomous quadrupedal robot manufactured by Ghost Robotics \cite{Ghost_Robotics}. It weighs approximately 26 kg. Vision 60 has 18 DOFs of which 12 leg DOFs are actuated. More specifically, each leg of the robot consists of a 1 DOF actuated knee joint with pitch motion and a 2 DOF actuated hip joint with pitch and roll motions. In addition, 6 DOFs are associated with the translational and rotational motions of the torso.

\noindent\textbf{Human Model:} The human model consists of a rigid tree structure with a torso link, including hands and head, and two identical legs terminating at point feet (see \cite{Hamed_Gregg_decentralized_control_IEEE_CST}). Each leg of the robot includes 3 actuated joints: a 2 DOF hip (ball) joint with roll and pitch motions and a 1 DOF knee joint in the sagittal plane. The model has 12 DOFs: 6 DOF for the translational and rotational motions of the torso and 6 DOF for the internal shape variables. The kinematic and dynamic parameter values for the links are
taken according to those reported in \cite{Human_Data} from a human cadaver study.

\noindent\textbf{Path Planning:} We consider an unleashed trotting gait $\mathcal{O}_{\textrm{ul}}^{d}$ for the dog robot at the speed of $1.2$ (m/s). To generate the gait, we make use of FROST (Fast Robot Optimization and Simulation Toolkit) --- an open-source toolkit for path planning of dynamic legged locomotion \cite{FROST,Ames_DURUS_TRO}. FROST makes use of direct collocation based trajectory optimization. In particular, it utilizes the Hermite-Simpson collocation approach to translate the path planning problem into a finite-dimensional nonlinear programming (NLP) that can be effectively solved with state-of-the-art NLP tools such as IPOPT. A desired periodic bipedal gait $\mathcal{O}_{\textrm{ul}}^{h}$ is designed for the locomotion of the human model at the speed of $1.1$ (m/s). We intentionally design the human gait to be slower than that of the dog to show that using the proposed control strategy, there can be a common speed leashed gait.

\noindent\textbf{Local Baseline Controllers:} Using the semidefinite programming algorithm of \cite{Hamed_Buss_Grizzle_BMI_IJRR,Hamed_Gregg_decentralized_control_IEEE_CST}, we synthesize the virtual constraint controllers of \eqref{HZD_controllers} in an offline manner to exponentially stabilize the unleashed gaits for the dog and human models. In particular, the algorithm looks for the optimal outputs to be regulated such that the stability condition in Assumption \ref{Transvsersal Stable Periodic Orbits} is satisfied. We further do not consider the full state stability for the human gait. Instead, we consider the \textit{stability modulo yaw} \cite[Section 6.5]{Hamed_Buss_Grizzle_BMI_IJRR} to have a model of visually impaired people locomotion. We remark that the dog robot together with the leash structure will have the responsibility to stabilize the yaw motion for itself as well as the human. The leash baseline controller is further designed to keep the human in the safe zone of $[1.25,1.75]$ (m). Figures \ref{COM_trajectories_no_leash} and \ref{COM_trajectories_with_leash} depict the robot and human center of mass (COM) trajectories in the $xy$-plane without and with using the leash structure, respectively. We remark that without the leash, the human gait does \textit{not} have the yaw stability (see Fig. \ref{COM_trajectories_no_leash}). However, utilizing the leash structure, the robot and human trajectories converge to a complex gait with the same speed while having the yaw stability (i.e., locomotion along the $x$-axis on which the yaw angle is zero) (see Fig. \ref{COM_trajectories_with_leash}).

\noindent\textbf{Obstacle Avoidance:} In order to demonstrate the power of the proposed hierarchical control algorithm, we consider a set of point obstacles $\mathscr{P}_{\alpha}^{o}$ for some $\alpha$ in the discrete set $\mathcal{I}^{o}$. The critical points on the robot and human (i.e., $\mathscr{P}^{d}_{\beta}$ and $\mathscr{P}_{\gamma}^{h}$) are then chosen as the hip points. In the first simulation, we consider $11$ obstacles around the steady-state trajectory of Fig. \ref{COM_trajectories_with_leash}. Without employing the real-time QP-based modification, the robot and human COM can hit the obstacles. In particular, Fig. \ref{COM_trajectories_with_leash} illustrates an undershoot around $-0.3$ (m) along the $y$-axis for the human COM that can easily collide with the obstacle located there in Fig. \ref{COM_trajectories_point_obstacles_CBFs}. However, utilizing the hierarchical control algorithm with QP running at 1kHz, the robot and human trajectories are locally modified around the steady-state gait such that the safety critical conditions are satisfied (see Fig. \ref{COM_trajectories_point_obstacles_CBFs}). The time profile for the robot's yaw and roll motions to accommodate the obstacles is depicted in Figs. \ref{Yaw_Roll_Motion_pointobstacles_CBFs:a} and \ref{Yaw_Roll_Motion_pointobstacles_CBFs:b}. The performance of the closed-loop hybrid system in the presence of a more-dense set of obstacles is shown in \ref{COM_trajectories_point_obstacles_CBFs_moredense}. Figure \ref{Snapshots} illustrates the snapshots of the robot and human locomotion around the obstacles. Animations can be found online \cite{YouTube_CooperativeLoco}.

\begin{figure*}[!t]
\vspace{1em}
\includegraphics[width=1\linewidth]{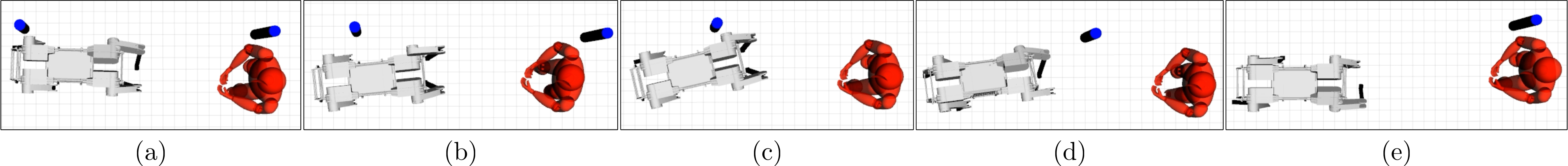}
\vspace{-1.4em}
\caption{Simulation snapshots illustrating the evolution ((a) to (e)) of Vision 60 and human trajectories on having a close encounter with the obstacles. The visualization does not illustrate the actuated leash. However, it's effect is clearly demonstrated by the augmented human trajectory in figures (see \cite{YouTube_CooperativeLoco} for the animation).}
\label{Snapshots}
\vspace{-2em}
\end{figure*}

%%%%%%%%%%%%%%%%%%%%%%%%%%%%%%%%%%%%%%%%%%%%%%%%%%%%%%%%%%%%%%%%%%%%%%%%%%%%%%%%

\vspace{-1em}
\section{CONCLUSION}
\vspace{-0.1em}

This paper presented a formal method towards 1) addressing complex hybrid dynamical models that describe cooperative locomotion of guide legged robots and humans and 2) systematically designing hierarchical control algorithms that enable stable and safe collaborative locomotion in the presence of discrete obstacles. At the higher level of the proposed control strategy, baseline controllers are assumed for the robotic dog and the leash structure. The robot baseline controller is developed based on HZD approach to asymptotically stabilize a pre-designed unleashed gait for the quadrupedal robot. The leash baseline controller is further developed to keep the human in a safe distance from the dog while following it. The existence and exponential stability of leashed gaits for the complex model are investigated via the Poincar\'e return map. At the lower level, a real-time QP is solved to modify the baseline controllers for the robot as well as the leash to ensure safety (i.e., obstacles avoidance) via CBFs. The power of the analytical approach is validated through extensive numerical simulations of a complex hybrid model with $60$ state variables and $20$ control inputs that represents the cooperative locomotion of Vision 60 and a human model. We considered an unleashed trotting gait for the dog and a bipedal gait for the human, where the dog gait is assumed to be faster. We further assumed that the human gait does not have yaw stability. It is shown that using the proposed control strategy, the dog and human can reach a common speed for the leashed motion. Moreover, we demonstrated that the robot can stabilize the yaw motion for the human model. The proposed approach can locally guarantee safety around pre-designed unleashed trajectories. The QP framework can significantly reduce the overshoot and undershoot in the human COM trajectories for following the guide robot. For future research, we will improve control algorithms to address sharp turns around corners and obstacles. We will extend the approach to consider dynamic obstacles. We will also investigate robust hierarchical approaches to address cooperative locomotion over uneven terrains.

%%%%%%%%%%%%%%%%%%%%%%%%%%%%%%%%%%%%%%%%%%%%%%%%%%%%%%%%%%%%%%%%%%%%%%%%%%%%%%%%

\vspace{-0.8em}
\bibliographystyle{IEEEtran}
\bibliography{references}

\end{document}